\documentclass[a4paper,twoside,10pt]{article}

\usepackage[T1]{fontenc}
\usepackage{times}
\usepackage{amsmath,amsfonts,amssymb,amsthm,euscript,pifont}
\usepackage{mathrsfs}
\usepackage{authblk}

\usepackage{color}
\usepackage[colorlinks=true]{hyperref}
\hypersetup{
    linkcolor = {blue},
    citecolor = {blue}
    }
\usepackage{fullpage}
\usepackage{xspace}

\usepackage{color}

\usepackage{dsfont}
\newcommand{\1}{\mathds{1}}

\newcommand{\er}{\mathbb R}
\newcommand{\nat}{\mathbb N}

\newcommand{\hyp}{({\it H})\xspace}
\newcommand{\hypi}{{\it H-(i)}\xspace}
\newcommand{\hypii}{{\it H-(ii)}\xspace}
\newcommand{\hypiii}{{\it H-(iii)}\xspace}

\newtheorem{theorem}{Theorem}[section]
\newtheorem{corollary}[theorem]{Corollary}

\newtheorem{lemma}[theorem]{Lemma}
\newtheorem{proposition}[theorem]{Proposition}

\newtheorem{remark}[theorem]{Remark}

\newcommand{\Cc}{\mathcal C}
\newcommand{\Dd}{\mathcal D}
\newcommand{\Ee}{\mathcal E}

\newcommand{\Pp}{\mathcal P}
\newcommand{\Tt}{\mathcal T}

\newcommand{\DD}{\mathbb D}
\newcommand{\EE}{\mathbb E}
\newcommand{\QQ}{\mathbb Q}
\newcommand{\NN}{\mathbb N}
\newcommand{\PP}{\mathbb P}

\newcommand{\dist}{{\mbox{dist}}}

\newcommand{\nd}{n_{\mathcal{D}}}
\newcommand{\UUU}{{U}}

\newcommand{\trunc}{{\beta}}
\newcommand{\weight}{{\omega}}
\newcommand{\minMaxwell}{{\underline{m}}}
\newcommand{\maxMaxwell}{{\overline{M}}}
\newcommand{\dens}{h}

\numberwithin{equation}{section}

\title{Particle approximation for Lagrangian Stochastic Models with specular boundary condition}

\author[1]{Mireille Bossy\thanks{mireille.bossy@inria.fr}}
\author[2]{Jean-Fran\c{c}ois Jabir \thanks{jean-francois.jabir@uv.cl. The second author acknowledges the FONDECYT Iniciaci\'on en Investigaci\'on Project N\textordmasculine 11130705 and The Iniciativa Científica Milenio N\textordmasculine 130062 for its support.}}
\affil[1]{TOSCA Laboratory, INRIA Sophia Antipolis -- M\'editerran\'ee, France}
\affil[2]{CIMFAV, Facultad de Ingenier\'ia, Universidad de Valpara\'iso, Chile}

\date{\today}

\begin{document}
\maketitle
\begin{abstract}
In this paper, we prove a particle approximation, in the sense of the propagation of chaos, of a Lagrangian stochastic model submitted to specular boundary condition and satisfying the mean no-permeability condition.
\end{abstract}


\section{Introduction}

In this paper, we construct a particle approximation of the following Lagrangian stochastic model $(X,U)$ on a finite time interval $[0,T]$,  submitted to specular reflections at the boundary of a compact smooth  domain $\Dd$ of $\er^d$:
\begin{equation}
\label{eq:ConfinedLagrangsystem}
\left\{\begin{aligned}
&X_t=X_0+\int_0^t \UUU_s ds,\quad
\UUU_t=\UUU_0+\int_0^t B[X_s;\rho(s)]ds + \sigma W_t+K_t,\\
&K_t=-2\sum_{0<s\leq t}\left(\UUU_{s^-}\cdot n_{\Dd}(X_s)\right)n_{\Dd}(X_s)\1_{\{X_s\in\partial\Dd\}},\\
&\rho(t)\,\mbox{is the Lebesgue density of the law of }(X_t,\UUU_t) \mbox{ for } t\in (0,T].
\end{aligned}\right.
\end{equation}
The initial condition $(X_0,\UUU_0)$ is distributed according to a given probability measure $\mu_0$, and is  independent to the $\er^d$-Brownian motion $(W_t;t\in[0,T])$, $\nd$ is the outward normal unit vector of the smooth boundary $\partial\Dd$.
We are considering Lagrangian stochastic model, this means that the dependencies in $x$ of the coefficients in the velocity equation \eqref{eq:ConfinedLagrangsystem} are expressed as a conditional expectation with respect to the event $\{X_t=x\}$. Here  the drift component $B[x;\rho(t)]$ is a version of the conditional expectation $\EE\left[b(\UUU_t)\,|\,X_t=x\right]$. Thus given a kernel $b$,  $B$ is defined for $(x,\gamma) \in \Dd\times L^1(\Dd\times\er^d)$ as
\begin{equation}\label{eq:NonlinearDrift}
B[x;\gamma]=\left\{
\begin{aligned}
&\frac{\int_{\er^d}b(v)\gamma(x,v)dv}{\int_{\er^d}\gamma(x,v)dv}\,\mbox{ if}\,\int_{\er^d}\gamma(x,v)dv\neq 0,\\
&0\,\mbox{otherwise.}
\end{aligned}
\right.
\end{equation}
A particle approximation of Lagrangian Stochastic Models (SLM) like \eqref{eq:ConfinedLagrangsystem} has been studied without the confining jump term $(K_t; t\in[0,T])$ in \cite{BoJaTa-10}, and in the particular one-dimensional confinement case in \cite{BoJa-11}. The existence and uniqueness (in the weak sense) of a solution to \eqref{eq:ConfinedLagrangsystem} has been established in \cite{BoJa-15} for confined system in a smooth and compact domain $\Dd$, with the help of PDE techniques. Furthermore the unique solution $(X,\UUU)$ satisfies the boundary condition \eqref{eq:MeanNoPermeability}, and the  sequence of jump-times
\[
\tau_{n}=\inf\{\tau_{n-1}<t\leq T; X_t\in\partial\Dd\}\,\mbox{ for }\,n\geq 1,
 \quad\tau_{0}=0, 
\]
 is well-defined and strictly increasing with $n$ up to $T$, with the convention that $\inf \emptyset = T$.

Numerical algorithms for SLM are based on particle approximation methods
(see e.g. \cite{jabir-10b} and the references therein). Here we give a first convergence result of a particle approximation of \eqref{eq:ConfinedLagrangsystem}.
We study  the limit behavior of the interacting particle system $\{(X^{i,\epsilon,N},\UUU^{i,\epsilon,N},K^{i,\epsilon,N}),  i=1,\ldots,N\}$,
on a given probability space $(\Omega, \mathcal{F},
(\mathcal{F}_t;t\geq 0), \QQ)$ endowed  with independent copies  $\{(X^{i}_{0},\UUU^{i}_{0},(W^{i}_t;\,t\in[0,T])), i=1,\ldots,N\}$  of $(X_{0},\UUU_{0},(W_t;\,t\in[0,T]))$, defined as the solution to the following SDE system:
\begin{equation}\label{eq:Confined_Smoothed_Particle_System}
\left\{
\begin{aligned}
&X^{i,\epsilon,N}_{t} = X^{i}_{0} + \int_{0}^{t} \UUU^{i,\epsilon,N}_{s} ds,\\
&\UUU^{i,\epsilon,N}_{t} = \UUU^{i}_{0}+\int_{0}^{t}
B_\epsilon  [X^{i}_s ;\mu^{\epsilon,N}_s]ds
+ \sigma W^{i}_{t}+ K^{i,\epsilon,N}_{t},\\
&K^{i,\epsilon,N}_{t}=-2\sum_{\substack{0< s\leq t}}
\left(\UUU^{i,\epsilon,N}_{s^{-}}\cdot {n_{\Dd}(X^{i,\epsilon,N}_s)}\right)
{n_{\Dd}(X^{i,\epsilon,N}_s)}
{\1}_{\displaystyle \left\{X^{i,\epsilon,N}_{s}\in\partial\Dd\right\}},~i=1,\ldots ,N.
\end{aligned}
\right.
\end{equation}
where $\mu^{\epsilon,N}_t = \frac{1}{N}\sum_{i=1}^N\delta_{(X^{i,\epsilon,N}_t,\UUU^{i,\epsilon,N}_t)}$ is the marginal according to the two first canonical coordinates $(x(t), u(t))$  at a given time $t$ of the empirical measure $\frac{1}{N}\sum_{i=1}^N\delta_{
(X^{i,\epsilon,N},\UUU^{i,\epsilon,N},K^{i,\epsilon,N})}$ of the $N$-particles system. The drift $B_\epsilon[x;\gamma]$ is a smoothed version of $B[x;\gamma]$ in \eqref{eq:NonlinearDrift},  with the help of a family of mollifiers
$\phi_{\epsilon}(x):=\epsilon^{-d}\phi(\frac{x}{\epsilon})$, for some
$\phi\in\Cc^{1}_{c}(\Dd)$ such that  $\phi\geq 0$ and $\int_\Dd \phi_\epsilon(x)\,dx=1$. $B_\epsilon[x;\gamma]$  is defined for all $x\in \Dd$ and all $\gamma$ in the set of probability measures on $\overline{\Dd}\times\er^d$ as
\begin{equation}\label{eq:SmoothedNonlinearDrift}
(x,\gamma) \mapsto B_\epsilon[x;\gamma]=\frac{\int_{\Dd\times\er^d}b(v)\trunc_\epsilon(y)\phi_{\epsilon}(x-y)\gamma(dy,dv)}{\int_{\Dd\times\er^d}\trunc_\epsilon(y)\phi_{\epsilon}(x-y)
\gamma(dy,dv)+\epsilon}
\end{equation}
where
$\trunc_\epsilon(y)  = \1_{\{\dist(y,\partial \Dd) > \epsilon\}}$ cutoffs the support of $\gamma$ from a distance $\epsilon$ to $\partial \Dd$.

The existence and uniqueness in law for the solution of \eqref{eq:Confined_Smoothed_Particle_System} simply follow from Girsanov's transformation and from the wellposedness of the confined Langevin process (i.e. the case $b=0$, see Theorem 2.1 in  \cite{BoJa-15}). This step  only requires that $\Dd$ has a $C^3$-boundary and that the support of $\mu_0$ is included in $\Dd\times\er^d$.

Our main result is stated in Theorem \ref{thm:MainTheorem}: as the number of particle grows to infinity and the  mollifiers parameter $\epsilon$ goes to $0$, we prove that the particles
\eqref{eq:Confined_Smoothed_Particle_System} propagate the initial chaos with a limit law given by the  solution to \eqref{eq:ConfinedLagrangsystem}.

\begin{remark}[About the mean no-permeability condition on $\partial \Dd$]
In \cite{BoJa-15}, we prove that the solution to \eqref{eq:ConfinedLagrangsystem}  satisfies the so-called mean no-permeability condition: for $x\in \partial \Dd$,
\begin{equation}\label{eq:MeanNoPermeability}
\EE\left[(\UUU_t\cdot n_{\Dd}(X_t))|X_t=x\right]=0.
\end{equation}
Stochastic Lagrangian models have been introduced for complex simulation in Computational Fluid Dynamic (CFD).
The mean-Dirichlet boundary condition \eqref{eq:MeanNoPermeability} grounds the stochastic particle algorithm used to downscale simulations in CFD applications  (we refer to \cite{BoJaTa-10}\cite{jabir-10b} and the references therein for further details).

Notice that the particle approximation of \eqref{eq:MeanNoPermeability} (with the kernel $b(u,x) = u\cdot \nd(x)$) in a neighborhood of $\partial \Dd$ is still an issue, that  seems to require  the continuity of the  density of $(X_t,U_t)$ over $\overline{\Dd}$. Except in the  one-dimensional  case studied in \cite{BoJa-11},  and, to the best of our knowledge, such regularity result  is  unknown in the PDE literature on trace problems.
\end{remark}

\begin{remark}[About the sequence of passage times on $\partial \Dd$]
When $\Dd=(0,+\infty)\times\er^{d-1}$, $\sigma=1$ and $b=0$, the explicit expression of the joint law of $(\tau_n,\UUU_{\tau_n}, n\geq 1)$  enables to control uniformly the confinement process $(K^{i,\epsilon,N}_t;\,t\in [0,T])$.   For more general domain, we compensate the lack of such control by studying the trace problem for the density $\rho$.

Notice  that the estimate $(3.6)$ in \cite{BoJa-11} on the upper-bound of $\PP(\tau_n\leq T)$ contains a mistake, claiming that this probability decreases with $n$ uniformly in $T$. This shall be reformulated as follows: when the initial law $\mu_0$ has its support in $ ((0,+\infty)\times\er^{d-1})\times \er^d$, there exists a constant $C(T,m_0,\beta^*)$ depending on $T$, $m_0:=\int \sqrt{v}\mu_0(dy,dv)$ and the distance  $\beta^* = \sup \{\beta >0;  \text{supp}(\mu_0)  \subset([\beta,+\infty)\times\er^{d-1})\times\er^d \}$ of the  support of $\mu_0$ to $(\{0\}\times\er^{d-1})\times \er^d$, such that, for all $n\geq 5$,
\begin{equation}\label{eq:Correction}
\PP\left(\tau_n\leq T\right)\leq {C(T,m_0,}\beta^*)\dfrac{1}{2^n}.
\end{equation}
This clarification of the constant in front of ${1}/{2^n}$ in the  right hand side  does not impact the results in  \cite{BoJa-11}, as we worked with fixed $T$. For completeness we give a short  proof of \eqref{eq:Correction} in Appendix \ref{sec:A2}.
\end{remark}

\paragraph{Notation.}
$\mathcal{E}$ denotes the set of paths $\Cc([0,T];\overline{\Dd})\times\DD([0,T];\er^d)\times\DD([0,T];\er^d)$.
 For all $t\in(0,T]$, we introduce the following sets: $Q_{t}:=(0,t)\times\Dd\times\er^{d}$,
\begin{align*}
&\Sigma^{+}_{T}:=\left\{(t,x,u)\in(0,T)\times\partial\Dd\times\er^{d}; (u\cdot \nd(x))>0\right\},\\
&\Sigma^{-}_{T}:=\left\{(t,x,u)\in(0,T)\times\partial\Dd\times\er^{d};(u\cdot \nd(x))<0\right\},\\
&\Sigma^{0}_T:=\left\{(t,x,u)\in(0,T)\times\partial\Dd\times\er^{d};(u\cdot \nd(x))=0\right\}.
\end{align*}
We set $|\Dd|:= \int_\Dd dx$. Denoting by $d\sigma_{\partial \Dd}$ the surface measure on $\partial\Dd$,  the product measure on $\Sigma_{T}:=\Sigma^+_T\cup\Sigma^0_T\cup\Sigma^-_T$ is  $d\lambda_{\Sigma_{T}}:=dt\otimes d\sigma_{\partial \Dd}(x)\otimes du$.
For a given positive weight function $\weight$ on $\er^d$, we define
the following weighted Sobolev spaces
\begin{align*}
&L^2(\weight;Q_T):=\{f:Q_T\rightarrow \er; \Vert f\Vert_{L^2(\weight;Q_T)}^2:={\int_{Q_T} \weight(u)f^2(t,x,u)\,dt\,dx\,du} <+\infty\},
\\
&V_1(\weight;Q_T):=\{f\in\Cc([0,T];L^2(\weight;\Dd\times\er^d));\\
&\hspace{4cm}\Vert f\Vert_{V_1(\weight;Q_T)}:=\max_{t\in[0,T]}\Vert f (t)\Vert_{L^2(\weight;\Dd\times\er^d)}+\Vert \nabla_u f\Vert_{L^{2}(\weight;Q_T)}<+\infty\},
\\
&L^{2}(\weight;\Sigma^{\pm}_{T}):=\big\{f:\Sigma^{\pm}_{T}\rightarrow \er;
\Vert f\Vert_{L^2(\weight;\Sigma^{\pm}_T)}^2:=\int_{\Sigma^{\pm}_T} \weight(u) |(u\cdot \nd(x))| f^2(t,x,u) d\lambda_{\Sigma_T}<+\infty \big\},
\\
&L^{2}(\weight;\Dd\times\er^d):=\big\{f:\Dd\times\er^d\rightarrow \er;
\Vert f\Vert^2_{L^2(\weight;\Dd\times\er^d)}:=\int_{\Dd\times\er^d} \weight(u) f^2(x,u) dx\,du<+\infty \big\}.
\end{align*}
$\mathcal{M}(E)$ denotes the set of probability measures on a measurable space $E$.  When this is not ambiguous, we will use $\|f\|_p$ for $\|f\|_{L^p(E)}$ with $1\leq p \leq +\infty$.
\section{Main results}

\paragraph{Hypotheses. }From now on, we assume that the domain $\Dd$, the distribution $\mu_{0}$ of $(X_{0},\UUU_{0})$,
and the kernel $b$ in~\eqref{eq:ConfinedLagrangsystem} satisfy the following hypotheses \hyp.
\begin{description}
\item[\hypi] ${\partial \Dd}$ is a compact $\Cc^3$ sub-manifold of $\er^d$. The initial measure $\mu_0$ has support in the interior of $\Dd\times\er^d$ and $\int_{\Dd\times\er^{d}}|u|^{2}\mu_{0}(dx,du)<+\infty$.  $\mu_{0}$ has a density $\rho_{0}$ in the weighted  space $L^{2}(\weight;\Dd\times\er^{d})$ with
\[
\weight(u)=(1+|u|^2)^{\frac{\alpha}{2}},\quad\mbox{ for  some }\alpha>{(d+3)}.
\]
\item[\hypii] $b:\er^{d}\longrightarrow \er^{d}$ is a bounded {continuous} function and $\sigma>0$.

\item[\hypiii] There exist $\underline{P}_{0}$, $\overline{P}_{0}:\er^{+}\longrightarrow \er^{+}$  in $L^1(\er^+)$ such that $u\mapsto \sqrt{(1+|u|)}\overline{P}_{0}(|u|) + \sqrt{\overline{P}_0(|u|)}\in L^2(\weight; \er^d)$, and
\begin{align*}
&0 < {\displaystyle \underline{P}_{0}(|u|)\leq\rho_{0}(x,u)\leq \overline{P}_{0}(|u|),~\mbox{a.e. on}~\Dd\times\er^{d}}.
\end{align*}
\end{description}
Notice that \hyp are slightly more restrictive than the hypotheses in \cite{BoJa-15} for the existence: here  $b$ is assumed continuous  to simplify some weak convergence arguments, and the weight function $\weight$ is  chosen in order to control $\int_{\er^d} |u|^3/{\weight(u)} du$.

\begin{theorem}\label{thm:MainTheorem}
Assume \hyp. Let $\PP$ be the law on $\mathcal{E}$ of $(X,U,K)$ defined in \eqref{eq:ConfinedLagrangsystem},  and let $\PP^{\epsilon,N}$ be the law of
$\{(X^{i,\epsilon,N},\UUU^{i,\epsilon,N},K^{i,\epsilon,N}),\,1\leq i\leq N\}$ defined in \eqref{eq:Confined_Smoothed_Particle_System}. Then $\PP^{\epsilon,N}$ is $\PP$-chaotic; namely,   for all $\{F_l,1\leq l\leq k\}$, $k\geq 2$, with $F_i\in \Cc_{b}(\Cc([0,T];\overline{\Dd})\times \DD([0,T];\er^d)\times\DD([0,T];\er^d))$, it holds that
\begin{equation*}
\lim_{\epsilon\rightarrow 0^+}\lim_{N\rightarrow +\infty}\langle F_1\otimes F_2\otimes \cdots \otimes F_k\otimes 1 \otimes 1 \otimes \cdots \otimes 1,\PP^{\epsilon,N}\rangle = { \prod_{l=1}^{k}}
\langle F_l,\PP\rangle.
\end{equation*}
\end{theorem}
Let us clarify the wellposedness of the weak solution to \eqref{eq:ConfinedLagrangsystem}. We summarize the results obtained from \cite{BoJa-15} in the following proposition.
\begin{proposition}[see  \cite{BoJa-15}]\label{prop:2014}
The law of the solution $(X,\UUU)$ to \eqref{eq:ConfinedLagrangsystem} is unique in the subset of  $\mathcal{M}(\Cc([0,T];\overline{\Dd})\times\DD([0,T];\er^{d})$)  that admits time-marginal densities $(\rho(t);\,t\in[0,T])$  in $L^{2}(\omega;\Dd\times\er^{d})$. Moreover, $(\rho(t);\,t\in[0,T])$  solves  in $V_1(\omega;Q_T)$  the PDE
\begin{equation}\label{eq:ConditionalMcKeanVlasov-Pde}
\left\{
\begin{aligned}
&\partial_t  \rho(t,x,u)+u\cdot \nabla_x \rho(t,x,u)-\frac{\sigma^2}{2}\triangle_u\rho(t,x,u)=-\left(B[x;\rho(t)]\cdot \nabla_u\rho(t,x,u)\right)\,\mbox{ in }\,Q_T,\\
&\rho(0,x,u)=\rho_0(x,u)~\mbox{ in }~\Dd\times\er^d,\\
&\gamma^{-}(\rho)(t,x,u)=\gamma^+(\rho)(t,x,u-2(u\cdot n_\Dd(x))n_\Dd(x))\,\mbox{ in }\,\Sigma^{-}_T,
\end{aligned}\right.
\end{equation}
where $\gamma^+(\rho)$ and $\gamma^-(\rho)$ are the trace functions of $\rho$, defined in $L^2(\weight;\Sigma^+_T)$ and $L^2(\weight;\Sigma^-_T)$ respectively, and satisfies the following energy estimate
\begin{equation}\label{EnergyEstimate}
\Vert \rho(t)\Vert^{2}_{L^2(\weight;\Dd\times\er^d)}+\sigma^2\int_0^t\Vert \nabla_u\rho(s)\Vert^{2}_{L^2(\weight;\Dd\times\er^d)}\,ds\leq \Vert \rho_0\Vert^2_{L^2(\weight;\Dd\times\er^d)}\left(1+C\exp(C t)\right),
\end{equation}
where $C>0$ depends only on $d$, $\alpha$ and $\Vert b\Vert_{\infty}$. In addition,
$\rho$ and its traces $\gamma^{\pm}(\rho)$ admit the following Maxwellian bounds: for a.e. $(t,x,u)\in Q_T$,
\begin{equation}\label{eq:MaxwellianBounds}
\exp(a_- t)\left(G_{\sigma}(t)* \underline{P}_0(|\cdot|)\right)^{\nu_-}(u)\leq \rho(t,x,u)\leq \exp(a_+ t)\left(G_{\sigma}(t)* \overline{P}_0(|\cdot|)\right)^{\nu_+}(u)
\end{equation}
and for $d\lambda_{\Sigma_T}$ a.e. $(t,x,u)\in \Sigma^{\pm}_{T}$,
\begin{equation}\label{eq:TraceBounds}
\exp(a_- t)\left(G_{\sigma}(t)* \underline{P}_0(|\cdot|)\right)^{\nu_-}(u)\leq \gamma^{\pm}(\rho)(t,x,u)\leq \exp(a_+ t)\left(G_{\sigma}(t)* \overline{P}_0(|\cdot|)\right)^{\nu_+}(u)
\end{equation}
where $G_{\sigma}(t)$ is the centered Gaussian density function with variance $\sigma^2 t$, $*$ stands for the convolution product, $a_{\pm}$, $\nu_{\pm}$ are constants depending only on $T,d,\alpha$ and $\Vert b\Vert_{\infty}$ and are such that $a_{-}<0$, $a_{+}>0$ and $\nu_{\pm}>0$.
\end{proposition}
Solution of Equation \eqref{eq:ConditionalMcKeanVlasov-Pde} and the related notion of trace functions stated in Proposition \ref{prop:2014} are understood in the weak (distributional) sense and we refer to \cite[Definition $1.1$ and Theorem $3.3$]{BoJa-15} for a  detailed existence result   formulation. For the sake of completeness, let us mention that the existence of $\gamma^+(\rho)$ and $\gamma^-(\rho)$ is directly related to the notion of trace problems for kinetic equations (see references in \cite{BoJa-15}), and their construction is granted, on $\Sigma^+_T$ and $\Sigma^-_T$ respectively, by a density argument related to the solution space $V_1(\weight,Q_T)$ and the smoothness of $\partial\Dd$.
Proposition \ref{prop:2014} is clearly also true when the drift $B[x;\cdot]$ is replaced by its smoothed version $B_\epsilon[x;\cdot]$, or by  a linear and bounded drift $V(t,x)$.  
The combination of Proposition \ref{prop:2014} with the following corollary allows to conclude on the uniqueness in law of the solution to \eqref{eq:ConfinedLagrangsystem}.
\begin{corollary}\label{coro:TraceAndMaxwellBounds}
Let $V\in L^\infty((0,T)\times\Dd)$ and assume \hyp. Any weak solution $(X_t,U_t;\,t\in[0,T])$ to
\begin{align}\label{XU-driftV}
\left\{
\begin{array}{l}
X_t=X_0+\int_0^t U_s\,ds, \\
U_t=U_0+\int_0^t V(s,X_s)\,ds+\sigma W_t -2\sum_{0<s\leq t}(U_{s^-}\cdot n_{\Dd}(X_s))n_{\Dd}(X_s)\1_{\{X_s\in\partial \Dd\}}
\end{array}
\right.
\end{align}
admits time-marginal densities $(\rho(t);\,t\in[0,T])$ {such that $\rho(t)$ is in $L^2(\omega,\Dd\times\er^d)$ for all $t\in[0,T]$}. The traces  $\gamma^{+}(\rho)$ and $\gamma^{-}(\rho)$ are in $L^2(\weight;\Sigma^+_T)$ and $L^2(\weight;\Sigma^-_T)$ respectively, and such that for all $f_1\in \Cc_c(\Sigma^+_T)$, $f_2\in \Cc_c(\Sigma^-_T)$,
\begin{equation}\label{eq:StochasticInterpretationTrace}
\begin{aligned}
&\EE\left[\sum_{n\in\NN}f_1(\tau_n,X_{\tau_n},U_{\tau_n^-})\1_{\{\tau_n\leq T\}}\right]=\int_{\Sigma^+_T} (u\cdot n_{\Dd}(x))f_1(t,x,u)\gamma^+(\rho)(t,x,u)d\lambda_{\Sigma_T},\\
&\EE\left[\sum_{n\in\NN}f_2(\tau_n,X_{\tau_n},U_{\tau_n})\1_{\{\tau_n\leq T\}}\right]=-\int_{\Sigma^-_T} (u\cdot n_{\Dd}(x))f_2(t,x,u)\gamma^-(\rho)(t,x,u)d\lambda_{\Sigma_T}.
\end{aligned}
\end{equation}
\end{corollary}
Notice that any solution to \eqref{eq:ConfinedLagrangsystem} is also a weak solution to \eqref{XU-driftV} for the bounded drift $V(t,x)=B[x;\rho(t)]$. Corollary \ref{coro:TraceAndMaxwellBounds} ensures that the time-marginal densities are in $L^2(\weight;Q_T)$ and Proposition \ref{prop:2014} allows to conclude the uniqueness of $\PP$ introduced in Theorem \ref{thm:MainTheorem}.

The proof of Corollary \ref{coro:TraceAndMaxwellBounds} is postponed in the appendix.  The rest of the paper is devoted to the proof of the propagation of chaos result.

Although  we give a particle approximation of the confined Lagrangian model, we are not able to use such approximation to construct a solution under lighter hypotheses than \hyp. In particular, we still have a deep use of the PDE analysis of the Fokker Planck equation. The main difficulty  resides in the uniform integrability result of the density traces, that we are able to show only with the strong Maxwellian bound tool.

\section{Proof of Theorem~\ref{thm:MainTheorem}}\label{Section_proof}

Equipped with the Skorokhod topology,  $\mathcal{E}$ is a Polish space. We denote by $(\mathcal{B}_{t};t\in[0,T])$ the filtration associated to the canonical process $(x(t),u(t),k(t);t\in [0,T])$ of $\mathcal{E}$.

The proof consists in the study of the double limits, first as $N$ tends to $\infty$, next as $\epsilon$ tends to $0$. Mainly, we will detail the two following steps:
\begin{proposition}\label{prop:ParticleLimit}Assume \hyp and fix  $\epsilon>0$.  The  SDE
\begin{equation}\label{eq:SmoothedConfinedLagrangsystem}
\left\{
\begin{aligned}
X^\epsilon_t&=X_0+\int_0^t \UUU^\epsilon_s ds,\quad
\UUU^\epsilon_t =\UUU_0+\int_0^t B_\epsilon[X^\epsilon_s;\mu^\epsilon(s)]ds + \sigma W_t+K^\epsilon_t,\\
K^\epsilon_t&=-2\sum_{0<s\leq t}\left(\UUU^\epsilon_{s^-}\cdot n_{\Dd}(X^\epsilon_s)\right)n_{\Dd}(X^\epsilon_s)\1_{\{X^\epsilon_s\in\partial\Dd\}}, \quad \mu^\epsilon(t)=\mbox{Law}(X^\epsilon_t,\UUU^\epsilon_t),
\end{aligned}
\right.
\end{equation}
has a unique weak solution, and we denote by $\PP^\epsilon$ the law  on $\mathcal{E}$ of $(X^\epsilon, U^\epsilon,K^\epsilon)$.   Then $\{\PP^{\epsilon,N};\,N\geq 1\}$ is $\PP^{\epsilon}$-chaotic; namely
for all  $k\geq 2$ and all $(F_{l},1\leq l\leq k)$ of functions in
$\Cc_{b}(\mathcal{E})$,
\begin{equation}\label{eq:SmoothedPropagChaos}
\lim_{N\rightarrow +\infty}\langle F_{1}\otimes \cdots F_{k}\otimes 1 \otimes 1 \otimes \cdots \otimes 1,\PP^{\epsilon,N}
\rangle= { \prod_{l=1}^{k}}
\langle\PP^{\epsilon},F_{l}\rangle.
\end{equation}
\end{proposition}

\begin{lemma}\label{lem:verification_iii_ii}
For any  converging subsequence   of $\{\PP^\epsilon;\,\epsilon>0\}$ (that we still denote by $\PP^\epsilon$), the sequence of time-marginals densities $\{\rho^\epsilon(t)=\PP^\epsilon\circ(x(t),u(t))^{-1}; \epsilon>0\}$ converges in $L^1(Q_T)$ and in $L^2(\weight;Q_T)$ to the time-marginals densities $\{\rho(t)=\QQ\circ(X_t,U_t)^{-1}\}$ of the solution to  \eqref{eq:ConfinedLagrangsystem}, when $\epsilon$ tends to $0$.
\end{lemma}

Proposition \ref{prop:ParticleLimit} has its analog in \cite{BoJa-11}. But now, the fact that the jump term is a finite variation process is not for free in the proof, as it is no more an increasing process.

Notice also that even if we consider constant diffusion process, the mild-equation tool that we strongly used in \cite{BoJaTa-10} and \cite{BoJa-11} is useless here, as the  controls we have on the semigroup derivative of the Lagrangian process are only for the $L^2$-norm.

\subsection{The limit as $N$ tends to $\infty$}

\begin{proof}[Proof of Proposition \ref{prop:ParticleLimit}]
The wellposedness of Equation \eqref{eq:SmoothedConfinedLagrangsystem} directly derives from Proposition \ref{prop:2014} (replacing $B$ by $B^\epsilon$) combining with Corollary \ref{coro:TraceAndMaxwellBounds}. We only have to prove the propagation of chaos result.

The verification of the Aldous's  criterion  for the tightness of the family $(\PP^{\epsilon,N}; N>0)$  is a  straightforward adaptation of Lemma 4.4 in \cite{BoJa-11}.    This  ensures  that the sequence    $\{\pi^{\epsilon,N} =\mbox{Law}(\frac{1}{N}\sum_{i=1}^N\delta_{
(X^{i,\epsilon,N},\UUU^{i,\epsilon,N},K^{i,\epsilon,N})});\,N\geq 1\}$ is tight on  $\mathcal{M}(\mathcal{E})$.

We check that all limit points of
$\{\pi^{\epsilon,N};\,N\geq 1\}$ have full measure on the
set of probability measures under which the canonical process $(x(t),u(t),k(t);\,t\in[0,T])$  satisfies  \eqref{eq:SmoothedConfinedLagrangsystem}.  We denote by  $\pi^{\epsilon,\infty}$ the limit of a converging subsequence of $\left\{\pi^{\epsilon,N};\,N\geq 1\right\}$ that we still index by $N$ for simplicity.

Following  Lemma 4.6 in \cite{BoJa-11},  it is not difficult to see that, for $\pi^{\epsilon,\infty}$-a.e. $m\in\mathcal{M}(\mathcal{E})$ with $(m(t):=m\circ (x(t),u(t))^{-1};t\in[0,T])$, the process
\begin{align}\label{eq:levy}
w_t:=\frac{1}{\sigma}\left(u(t) - u(0) -k(t) -\int_{0}^{t}B_{\epsilon}[x(s);m(s)]\,ds\right),\,t\in(0,T],
\end{align}
is a $\er^d$-Brownian motion under $m$.

The remaining point is the identification of the jump process $k$ that we detail in the following lemma.
\begin{lemma}\label{lem:JumpIdentification}
The three following properties hold true $\pi^{\epsilon,\infty}$-a.e, $m\in\mathcal{M}(\Ee)$, $m$-a.s.:
\item{(a)} For all jump times $t\in[0,T]$ of $u$, $\triangle u(t) = -2 (u(t^{-})\cdot n_{\Dd}(x(t)))n_{\Dd}(x(t))$.
\item{(b)} $(k(t);\,t\in[0,T])$ is a finite variation process, and the related measure $|k|$ defined on $[0,T]$ satisfies
\[
|k|(t)=\int_{0}^{t}\1_{\displaystyle\left\{s\geq 0; x(s)\in\partial\Dd\right\}}\,d|k|(s),~\forall~t\in[0,T].
\]
\item{(c)} The set $\{t\in[0,T]; x(t)\in\partial\Dd \}$ is at most countable.
\end{lemma}

The properties \textit{(b)}  and \textit{(c)} above imply that $(k(t);t\in[0,T])$ is a pure jump process. In addition, since
the paths of the process $(u(t)-k(t);\,t\in[0,T])$ are continuous by \eqref{eq:levy}, the jump times and the jump length of $(k(t),t\in[0,T])$ and
$(u(t),t\in[0,T])$ are a.s. undistinguishable. Therefore \textit{(a)} ensures that
$\pi^{\epsilon,\infty}$-a.e. $m\in\mathcal{M}(\mathcal{E})$, $m$-a.s.,
$k(t)=-2\sum_{0<s\leq t}(u(s^{-})\cdot n_{\Dd}(x(s)))n_{\Dd}(x(s))\1_{\{x_{s}\in\partial\Dd\}}$ for all $t\in[0,T]$.

The uniqueness in law for the solution of  \eqref{eq:SmoothedConfinedLagrangsystem} ensures that all converging subsequences of $\{\pi^{\epsilon,N},N\geq 1\}$
tend to $\delta_{\left\{\PP^{\epsilon}\right\}}$,  and enables us to conclude on  the propagation of chaos property \eqref{eq:SmoothedPropagChaos}.
\end{proof}

\begin{proof}[Proof of Lemma~\ref{lem:JumpIdentification}]
The proof mainly follows  the proof of Lemma 4.8 in \cite{BoJa-11}.
We only need to take care about points \textit{(b)} as in the multi-dimensional case,  the jump process $k$ is no more an increasing  process.

For $\Pp_{t}=\{p=\{0\leq t_1\leq \ldots\leq t_l\leq t\};  l\geq 1\}$, the set of  all partitions of the interval $[0,t]$, the total variation process  related to $k$ is defined as
\[
|k|(t):=\sup_{p\in\Pp_t} \sum_{i=0}^{l-1}\left|k(t_{i+1})-k(t_{i})\right|.
\]
First we prove that $\pi^{\epsilon,\infty}$-a.e. $m\in\mathcal{M}(\mathcal{E})$, $m$-a.s.,  $|k|(T)<+\infty$.
We replicate some arguments of Sznitman \cite{ASznitman1984} and introduce the sets
\[
F_M:=\Big\{(x,u,k)\in\mathcal{E}; |k|(T)\leq M,\,\int_{0}^{T}\dist(x(s),\partial\Dd)\,d|k|(s) = 0\Big\}, ~G^\eta_M:=\left\{m\in\mathcal{M}(\mathcal{E}); m(F_M)\geq 1-\eta\right\}.
\]
Let us show that for all $\eta >0$, $\lim_{M\rightarrow +\infty} \pi^{\epsilon,\infty}(G^{\eta}_M) =1$. 
Since $F_M$ is a closed subset of $\Ee$ (see below), $G^\eta_M$ is closed  for the weak topology on $\mathcal{M}(\mathcal{E})$. Therefore
\[
\pi^{\epsilon,\infty}(G^\eta_M)\geq \limsup_{N\rightarrow +\infty}\pi^{\epsilon,N}(G^\eta_M)=\limsup_{N\rightarrow +\infty}
\QQ(\{\overline{\mu}^{\epsilon,N}(F_M)\geq 1-\eta\}).
\]
Denoting $F_M^c$ the complement of $F_M$ on $\mathcal{E}$, we have
\[
\QQ(\{\overline{\mu}^{\epsilon,N}(F_M)\geq 1-\eta\})=1-\QQ(\{\overline{\mu}^{\epsilon,N}(F_M^c)> \eta\}).
\]
Then applying two times the Chebyshev's inequality, and using the exchangeability of the particles,
\[
\QQ(\{\overline{\mu}^{\epsilon,N}(F_M^c)> \eta\})\leq \frac{1}{\eta}\EE_\QQ\left[\langle\1_{\{F_M^c\}},\overline{\mu}^{\epsilon,N}\rangle\right]
\leq \frac{1}{M\eta}\EE_\QQ\left[|K^{1,\epsilon,N}|_T\right].
\]
Owing to Lebesgue's monotone convergence theorem, we have
\begin{align*}
&\EE_\QQ\left[|K^{1,\epsilon,N}|_T\right]
=\sup_{\mathcal{P}_{T}} \sum_{m=0}^{l-1}\EE_\QQ\left|\sum_{t_m<s\leq t_{m+1}}-2\left(\UUU^{1,\epsilon,N}_{s^-}\cdot n_\Dd(X^{1,\epsilon,N}_s)\right)n_\Dd(X^{1,\epsilon,N}_s)
\1_{\{X^{1,\epsilon,N}_s\in\partial\Dd\}}\right|.
\end{align*}
And, by the trace representation in Corollary  \ref{coro:TraceAndMaxwellBounds},
\begin{align*}
\EE_\QQ\left[|K^{1,\epsilon,N}|_T\right]\leq 2 \sup_{\mathcal{P}_{T}} \sum_{m=0}^{l-1}\int_{t_m}^{t_{m+1}}
\int_{\partial\Dd\times\er^d}|(u\cdot n_\Dd(x))|^2\gamma^{-}({\rho^{1,\epsilon,N}})(s,x,u)d\lambda_{\Sigma_T}\\
=\int_{\Sigma^-_T}|(u\cdot n_\Dd(x))|^2\gamma^{-}(\rho^{1,\epsilon,N})(s,x,u)d\lambda_{\Sigma_T}.
\end{align*}
Since $\gamma^-(\rho^{1,\epsilon,N})$ is bounded  in $L^2(\weight;\Sigma^-_T)$  uniformly w.r.t $N$, and $\int_{\er^d} \frac{|(u\cdot n_{\Dd}(x))|^3}{\weight(u)} du \leq \int_{\er^d} \frac{|u|^3}{\weight(u)} du< +\infty$,
\begin{align*}
\int_{\Sigma^-_T}|(u\cdot n_\Dd(x))|^2& \gamma^{-}(\rho^{1,\epsilon,N})(s,x,u)d\lambda_{\Sigma_T} \\
& \leq \sqrt{\int_{\Sigma^-_T} \frac{|u|^3}{\weight(u)}\,d\lambda_{\Sigma_T}}\sqrt{\int_{\Sigma^-_T} |(u\cdot n_\Dd(x))| \weight(u) |\gamma^{-}(\rho^{1,\epsilon,N})|^2(s,x,u)d\lambda_{\Sigma_T}}< +\infty.
\end{align*}

It follows that $\limsup_{N\rightarrow +\infty}\EE_\QQ\left[|K^{1,\epsilon,N}|_T\right]<+\infty$, so that $\lim_{M\rightarrow +\infty} \pi^{\epsilon,\infty}(G^{\eta}_M) =1$. Letting $\eta$ tends to $0$, we also conclude that for $\pi^{\epsilon,\infty}$-a.e. $m$, $m(\cup_{M>0}F_M) =1$ which means that $|k|(T)<+\infty$ a.s.

We prove now that $F_M$ is closed. Let us consider a sequence  $\{\zeta^n = (x^n,u^n,k^n), n\in\nat\}$ in that subset, converging  to $\zeta=(x,u,k)$ in $\mathcal{E}$ according to the Skorokhod topology; namely (see e.g. \cite[Theorem 1.14, Chapter 6]{JaSh-02}) there exists a sequence $\{\lambda_{n}, n\in\nat\}$  of continuous increasing functions on $[0,T]$ such that  for all $n$,  $\lambda_{n}(0)=0$, $\lambda_{n}(T)=T$, $\lim_{n\rightarrow +\infty}\sup_{t\in[0,T]}|\lambda_{n}(t)-t|=0$, and
\begin{equation}\label{eq:SkorokhodTopo}
\lim_{n\rightarrow+\infty} \sup_{t\in[0,T]}|{\zeta}^{n}(\lambda_{n}(t))-{\zeta}(t)|=0,
\mbox{ ~and for all }t\in[0,T] ~\lim_{n\rightarrow+\infty}
|\Delta{\zeta}^{n}(\lambda_{n}(t))-\Delta{\zeta}(t)|=0.
\end{equation}5
As \eqref{eq:SkorokhodTopo} implies that $k^n\circ \lambda_n$ and $\triangle k^n\circ \lambda_n$ converge respectively to $k$ and $\triangle k$ uniformly in $[0,T]$, it further ensures that the sequence of measures $d|k^n\circ\lambda_n|$ converges weakly to $d|k|$ on $[0,T]$. Indeed, $\lambda_n$ being a time change, let us first remark that the measure $d|k^n\circ\lambda_n|$ coincides with the pushforward measure  $d(\lambda^{-1}_n\sharp |k^n|)$, such that for all $s,t\in [0,T]$, $(\lambda^{-1}_n\sharp |k^n|)([s,t]) = |k^n|(\lambda_n(s),\lambda_n(t))$.
Since  $\sup_n |k^n|(T)  = \sup_n |k^n\circ\ \lambda_n |(T) \leq M$, there exists a converging subsequence $\{|k^{n_\ell}\circ\lambda_{n_\ell}|,\,\ell\in\NN\}$. From this sequence,  let us further extract a subsequence
  $\{|k^{n_L}\circ\lambda_{n_L}|,\,L\in\NN\}$ such that
$\sup_{[0,T]}|k^{n_L}\circ\lambda_{n_L}-k|\leq \frac{1}{L^2}$. Then, for all continuous function $f:[0,T]\rightarrow\er$,  for the partition $0=t_0\leq t_1\leq \cdots \leq t_{L-1}\leq t_L=T$ of $[0,T]$ such that $|t_{m+1}-t_m|\leq T/L$, we have
\begin{align*}
&\left|\int_0^T f(s)d|k|(s)-\int_0^Tf(s)d|k^{n_L}\circ\lambda_{n_L}|(s)\right|\\
&\leq \left|\int_0^T f(s)d|k|(s)-\sum_{m=0}^{L-1}f(t_m)\left|k(t_{m+1})-k(t_m)\right|\right|\\
&\quad +\left|\sum_{m=0}^{L-1}f(t_m)\left|k(t_{m+1})-k(t_m)\right|-\sum_{m=0}^{L-1}f(t_m)\left|k^{n_L}(\lambda_{n_L}(t_{m+1}))-k^{n_L}(\lambda_{n_L}(t_{m}))\right|\right|\\
&\quad +\left|\sum_{m=0}^{L-1}f(t_m)\left|k^{n_L}(\lambda_{n_L}(t_{m+1}))-k^{n_L}(\lambda_{n_L}(t_{m}))\right|-\int_0^Tf(s)d|k^{n_L}\circ\lambda_{n_L}|(s)\right|.
\end{align*}
The first and third terms in the right hand side tend to $0$ as $L\rightarrow + \infty$ by continuity of $f$. Since the second term is bounded from above by $2\max_{t\in[0,T]}|f(t)|/L$, we conclude that for any converging subsequence of $|k^n\circ \lambda_n|$ we can extract a subsequence which converges to $|k|$. This implies the weak convergence of $|k^n\circ \lambda_n|$ towards $|k|$.
 Next,
 since $t\mapsto \eta_n(t) :=\dist(x^{n}(t),\partial\Dd)$ and $t\mapsto \eta(t) := \dist(x(t),\partial\Dd)$ are  continuous and $\int_{0}^{T}\eta_n (s)d|k^n|(s)=0$,
\begin{align*}
\left|\int_{0}^{T}\eta (s)d|k|(s)\right|  = & \left | \int_{0}^{T}\eta (s)d|k|(s) -  \int_{0}^{T}\eta_n (s)d|k^n|(s)\right | \\
\leq & \left | \int_{0}^{T}\eta (s)d|k|(s) -  \int_{0}^{T}\eta (s)d|k^n\circ\lambda_{n}|(s)\right |  +  \left |  \int_{0}^{T}\eta (s)d|k^n\circ\lambda_{n}|(s) -  \int_{0}^{T}\eta_n (s)d|k^n|(s) \right |.
\end{align*}
The first term tends to 0, by the weak convergence of $|k^n\circ\lambda_{n}|$ to $|k|$. For the second one, using the change of variable in the second  integral
\[ \int_{0}^{T}\eta_n (s)d|k^n|(s)  = \int_{0}^{T}\eta_n (\lambda_n(\lambda_n^{-1}(s)))d|k^n|(s)  =  \int_{0}^{T}\eta_n(\lambda_{n}(s))  d \lambda_n^{-1}\sharp|k^n| (s)  =  \int_{0}^{T}\eta_n(\lambda_{n}(s))  d|k^n \circ \lambda_{n}| (s),\]
we get
$$ \left |  \int_{0}^{T}\eta (s)d|k^n\circ\lambda_{n}|(s) -  \int_{0}^{T}\eta_n (s)d|k^n|(s) \right |  \leq \left(\sup_{s\in [0,T]} |\eta_n(\lambda_n(s)) - \eta(s)|\right) M $$
from which we conclude that $\left|\int_{0}^{T}\eta (s)d|k|(s)\right| =0$,  letting $n$ tend to infinity.
\end{proof}
\subsection{The limit as $\epsilon$ tends to $0$}

  The tightness of $\{\PP^\epsilon;\,\epsilon>0\}$ can be shown again by replicating the verification of the Aldous's criterion given in Lemma $4.4$ of \cite{BoJa-11}.

The main concern in that step is for the identification of the limit points.
With Lemma \ref{lem:verification_iii_ii}, we easily check that any limit $\PP^0$ of a converging subsequence   of $\{\PP^\epsilon;\,\epsilon>0\}$ is a weak solution to the  \eqref{eq:ConfinedLagrangsystem},  as it satisfies the following  martingale problem  conditions
\begin{description}
\item[]{(i)} $\PP^0\circ({x(0),u(0),k(0)})^{-1} = \mu_{0}\otimes \delta_0$, where
$\delta_0$ denotes the Dirac mass at 0 on $\er^d$, by hypothesis.
\item[]{(ii)} For all $t\in(0,T]$, $\PP^0\circ(x(t),u(t))^{-1}$ admits the positive
Lebesgue density $\rho(t)$.
\item[]{(iii)} For all $f\in\Cc^{2}_{b}(\er^{2d})$, the process
\end{description}
\begin{align}\label{Confined_boundedPope_Martingale_System}
\begin{aligned}
&f(x(t),u(t)-k(t))-f(x(0),u(0))-\int_{0}^{t}\left(u(s)\cdot\nabla_{x}f(x(s),u(s)-k(s))\right)ds\\
&\quad -\int_{0}^{t}\left[\left(B\left[x(s);\rho(s)\right]\cdot
\nabla_{u}f(x(s),u(s)-k(s))\right)+\frac{\sigma^2}{2}\triangle_{u}f(x(s),u(s)-k(s))\right]ds\\
&\mbox{is a continuous $\PP^0$-martingale w.r.t. the canonical filtration $\left(\mathcal{B}_{t};t\in [0,T]\right)$.}
\end{aligned}
\end{align}
Indeed, as observed in \cite{BoJa-11},  this is a direct consequence of the following convergence
\begin{equation}\label{L1Continuity}
\lim_{|h|,|\delta|\rightarrow 0}\limsup_{\epsilon\rightarrow 0^{+}}
\int_{\Dd\times\er^{d}}\left|\rho^{\epsilon}_{t}(x+h,u+\delta)-\rho^{\epsilon}_{t}(x,u)\right|dx\,du=0,~\forall~t\in(0,T],
\end{equation}
that can be immediately deduced from Lemma \ref{lem:verification_iii_ii}.
\begin{description}
\item[]{(iv)} $\PP^0$-a.s., the set $\{t\in[0,T]; {x(t)\in\partial\Dd}\}$ is at
most countable, and for all $t\in[0,T]$,
\begin{equation*}
k(t)=-2\sum_{0<s\leq t}(u(s^-)\cdot n_{\Dd}(x(s)))n_{\Dd}(x(s))\1_{\left\{x(s)\in\partial\Dd\right\}},
\end{equation*}
\noindent
since we can reproduce all the arguments of Lemma \ref{lem:JumpIdentification}., applying Corollary \ref{coro:TraceAndMaxwellBounds} again.
\end{description}
\begin{proof}[Proof of Lemma \ref{lem:verification_iii_ii}]
The existence of the time-marginal densities $(\rho^\epsilon(t);\,t\in[0,T])$ in $L^2(\weight;Q_T)$  follows immediately from Corollary \ref{coro:TraceAndMaxwellBounds}. Adapting Proposition \ref{prop:2014} to $B^\epsilon$, $\rho^{\epsilon}$ satisfies in $V_1(\weight;Q_T)$ the analog of the Fokker-Planck equation \eqref{eq:ConditionalMcKeanVlasov-Pde}, replacing $B$ by $B^\epsilon$.
Thus we observe that, for all $\epsilon>0$,
$R^{\epsilon}:=\rho^{\epsilon}-\rho$ satisfies in $V_1(\weight;Q_T)$
\begin{equation*}
\left\{
\begin{aligned}
&\partial_t R^{\epsilon}+u\cdot \nabla_x R^{\epsilon}-\frac{\sigma^2}{2}\triangle_u R^{\epsilon}=-\nabla_u\cdot \left(B_{\epsilon}[\cdot;\rho^{\epsilon}]\rho^{\epsilon}-B[\cdot;\rho]\rho\right)~\mbox{in }Q_T,\\
&R^{\epsilon}(0,x,u)=0\,\mbox{ in}\,\Dd\times\er^d,\\
&\gamma^{+}(R^{\epsilon})(t,x,u)=\gamma^-(R^{\epsilon})(t,x,u-2(u\cdot n_{\Dd}(x))n_{\Dd}(x))\,\mbox{ in}\,\Sigma^+_T.
\end{aligned}
\right.
\end{equation*}
Therefore,  applying \cite[Lemma 3.8]{BoJa-15} with $B=B_\epsilon[\cdot;\rho_\epsilon]$, $q(t,x,u)=\gamma^{-}(R^{\epsilon})(t,x,u-2(u\cdot n_{\Dd}(x))n_{\Dd}(x))$, $g=\left(B_\epsilon[\cdot;\rho^\epsilon]-B[\cdot;\rho])\cdot \nabla_u\rho\right)$,  for all $t\in(0,T]$,
\begin{align*}
&\Vert R^{\epsilon}(t)\Vert^2_{L^2(\weight;\Dd\times\er^d)}+\sigma^2\int_0^t \Vert \nabla_u R^{\epsilon}(s)\Vert^2_{L^2(\weight;\Dd\times\er^d)}\,ds\\
&=\int_{Q_t} \left(B_{\epsilon}[\cdot;\rho^{\epsilon}]\rho^{\epsilon}-B[\cdot;\rho]\rho\right)\cdot (R^{\epsilon}\nabla_u\weight+\weight\nabla_u  R^{\epsilon}) +\sigma^2\int_{Q_t}(\triangle_u\weight) (R^{\epsilon})^2.
\end{align*}
Since $|\triangle_u\weight(u)|+|\nabla_u\weight(u)|\leq C(\alpha,d)\weight(u)$, using Young Inequality,  it follows that
\begin{equation}\label{proofstep6}
\begin{aligned}
&\Vert R^{\epsilon}(t)\Vert^2_{L^2(\weight;\Dd\times\er^d)}+\frac{\sigma^2}{2}\int_0^t \Vert \nabla_u R^{\epsilon}(s)\Vert^2_{L^2(\weight;\Dd\times\er^d)}\,ds\\
&\leq \left(C(\alpha,d)\Vert b\Vert_{\infty}+\sigma^2 C(\alpha,d)+\frac{\Vert b\Vert_{\infty}}{2\sigma^2}\right)\int_0^{t}\Vert R^{\epsilon}(s)\Vert^2_{L^{2}(\weight;\Dd\times\er^d)}\,ds\\
&\quad+\left(C(\alpha,d)+\frac{1}{2\sigma^2}\right)\int_0^
t\Vert\rho(s)\left(B_{\epsilon}[\cdot;\rho^{\epsilon}(s)]-B[\cdot;\rho(s)]\right)\Vert^2_{L^{2}(\weight;\Dd\times\er^d)}\,ds.
\end{aligned}
\end{equation}
Now, observe that, according to the Maxwellian bounds in  Proposition \ref{prop:2014}, we can define finite positive constants
$$\maxMaxwell:=\sup_{\substack{(t,x)\in(0,T)\times\Dd}}\left(\int_{\er^d}\weight(v)|\rho(t,x,v)|^2\,dv\right) \mbox{ and  } \minMaxwell:=\inf_{\substack{(t,x)\in(0,T)\times\Dd}}\left(\int_{\er^d}\rho(t,x,v)\,dv\right),$$
with $\underline{m}>0$ and $\overline{M}<\infty$, from the Maxwellian bounds and \hypiii.
Then
\begin{align*}
&\int_0^t\Vert\rho(s)\left(B_{\epsilon}[\cdot;\rho^{\epsilon}(s)]-B[\cdot;\rho(s)]\right)\Vert^2_{L^{2}(\weight;\Dd\times\er^d)}\,ds
{\leq \maxMaxwell \int_{(0,t)\times\Dd}\left|B_{\epsilon}[x;\rho^{\epsilon}(s)]-B[x;\rho(s)]\right|^2\,dx\,ds}.
\end{align*}
By setting $\overline{\rho}(t,x):=\int_{\er^d}\rho(t,x,u)\,du$, $\overline{b\rho}(t,x):=\int_{\er^d}b(u)\rho(t,x,u)\,du$,
$\overline{\rho^\epsilon}(t,x):=\int_{\er^d}\rho^\epsilon(t,x,u)\,du$ and $\overline{b\rho^\epsilon}(t,x):=\int_{\er^d}b(u)\rho^\epsilon(t,x,u)\,du$
and for $*$ the convolution product, we have
\begin{align*}
&B_{\epsilon}[x;\rho^{\epsilon}(t)]-B[x;\rho(t)] =  \frac{\phi_{\epsilon} * \left(\trunc_{\epsilon} \overline{b\rho^{\epsilon}}\right)(t,x)}{\phi_{\epsilon} * (\trunc_{\epsilon}\overline{\rho^{\epsilon}})(t,x)+\epsilon} -\frac{\overline{b\rho}(t,x)}{\overline{\rho}(t,x)} \\
&=\phi_{\epsilon} * (\trunc_{\epsilon} \overline{b\rho^{\epsilon}})(t,x)
\left(\frac{\overline{\rho}(t,x)-\phi_{\epsilon} * \left(\trunc_{\epsilon} \overline{\rho^{\epsilon}}\right)(t,x)-\epsilon}{\left(\phi_{\epsilon} * \left(\trunc_{\epsilon} \overline{\rho^{\epsilon}}\right)(t,x)+\epsilon\right)\overline{\rho}(t,x)}\right)
+\frac{\phi_{\epsilon} * \trunc_{\epsilon} \overline{b\rho^{\epsilon}}(t,x)-\overline{b\rho}(t,x)}{\overline{\rho}(t,x)}\\
&=\frac{\phi_{\epsilon} * \left(\trunc_{\epsilon} \overline{b\rho^{\epsilon}}\right)(t,x)}{\phi_{\epsilon} * (\trunc_{\epsilon}\overline{\rho^{\epsilon}})(t,x)+\epsilon}
\left(\frac{\phi_{\epsilon} * \left(\trunc_{\epsilon}\overline{ \rho}-\trunc_{\epsilon} \overline{\rho^{\epsilon}}\right)(t,x)}{\overline{\rho}(t,x)}\right) +\frac{1}{\overline{\rho}(t,x)}
\left(\phi_{\epsilon} * \left(\trunc_{\epsilon}\left(\overline{b\rho^{\epsilon}}-\overline{b\rho}\right)\right)(t,x)\right)\\
&\quad +\frac{\phi_{\epsilon} * \left(\trunc_{\epsilon}\overline{ b\rho^{\epsilon}}\right)(t,x)}{\phi_{\epsilon} * (\trunc_{\epsilon}\overline{\rho^{\epsilon}})(t,x)+\epsilon}
\left(\frac{\overline{\rho}(t,x)-\phi_{\epsilon} * (\trunc_{\epsilon}\overline{\rho})(t,x)-\epsilon}{\overline{\rho}(t,x)}\right) +\frac{\phi_{\epsilon} * (\trunc_{\epsilon}\overline{b\rho})(t,x)-\overline{b\rho}(t,x)}{\overline{\rho}(t,x)},
\end{align*}
which gives
\begin{align*}
\int_{(0,t)\times\Dd}\left|B_{\epsilon}[x;\rho^{\epsilon}(s)]-B[x;\rho(s)]\right|^2\,dx ds
\leq \frac{8\Vert b\Vert^2_{\infty}}{\minMaxwell^2}\left[\int_{(0,t)\times\Dd\times\er^d}\left|\rho^\epsilon(s,x,u)-\rho(s,x,u)\right|^2{\,du}\,dx\,ds+\Tt_{\epsilon}\right],
\end{align*}
for $\Tt_{\epsilon}:=\Vert\phi_{\epsilon} * (\trunc_{\epsilon}\overline{\rho})-\overline{\rho}\Vert^2_{L^{2}((0,t)\times\Dd)}+\frac{1}{\Vert b\Vert^2_{\infty}}\Vert\phi_{\epsilon} * (\trunc_{\epsilon}\overline{b\rho})-\overline{b\rho}\Vert^2_{L^{2}((0,t)\times\Dd)} + t|\Dd| \epsilon$.
Coming back to \eqref{proofstep6} and using Gronwall's inequality, it follows that
\begin{equation*}
\Vert R^{\epsilon}(t)\Vert^2_{V_1(\weight;Q_t)}\leq C\Tt_{\epsilon},
\end{equation*}
for $C$ independent of $\epsilon$. Since $\rho$ is in $L^2(\omega;Q_T)$ and $\lim_{\epsilon\rightarrow 0}\trunc_{\epsilon}=1$ a.e. on $\Dd$, $\Tt_\epsilon$ tends to $0$ as $\epsilon$ tends to $0$. Hence
$\lim_{\epsilon\rightarrow 0}\Vert R^{\epsilon}\Vert_{V_1(\weight;Q_T)}=0$.
\end{proof}
\appendix
\section{Appendix}\label{appendix}
\subsection{Proof of Corollary \ref{coro:TraceAndMaxwellBounds}}
From the Riesz Representation Theorem,  it is sufficient to check that for all $t\in (0,T]$, there exists some constant $C>0$ such that
\begin{equation}\label{eq:weightedL2bound}
\left|  \EE_{\QQ}\left[\sqrt{\weight(\UUU_t)}\psi(X_t,\UUU_t)\right] \right| \leq C \Vert \psi\Vert_{L^2(\Dd\times\er^d)} \mbox{, for all } \psi\in L^2(\Dd\times\er^d).
\end{equation}
Without loss of generality,  let us assume that $\psi$ is nonnegative. Then 
\begin{align*}
\EE_{\QQ}[\sqrt{\weight(\UUU_t)}\psi(X_t,\UUU_t)]
&=\int_{\Dd\times\er^d}\EE_{{\QQ}}\left[\sqrt{\weight(\UUU_t)}\psi(X_t,\UUU_t)\big| (X_0,\UUU_0)=(x,u)\right]\rho_0(x,u)dxdu\\
&\leq |\Dd|\int_{\Dd\times\er^d}\EE_{{\QQ}}\left[\sqrt{\weight(\UUU_t)}\psi(X_t,\UUU_t)\big| (X_0,\UUU_0)=(x,u)\right]\frac{1}{|\Dd|}\overline{P}_0(|u|)dxdu\\
&  \quad =   \|\overline{P}_0\|_1  |\Dd| \EE_{\QQ}\left[Z_t \sqrt{\weight(v_t)}\psi(y_t,v_t)\right]
\end{align*}
where $(y,v)$ is defined as
\[
y_t=Y_0+\int_0^t v_s ds,\qquad v_t=V_0+\sigma W_t - 2\sum_{0<s\leq t}(v_{s^-}\cdot n_\Dd(y_s))n_\Dd(y_s)\1_{\{y_s\in\partial\Dd\}},
\]
with $(Y_0,V_0)$  distributed according to
$\tfrac{dx}{|\Dd|}\tfrac{\overline{P}_0(|u|)}{\|\overline{P}_0\|_1}du$, and where $Z_t=\exp(-\tfrac{1}{\sigma} \int_0^t V(s,y_s)dW_s-\tfrac{1}{2\sigma^2}\int_0^t | V(s,y_s)|^2ds)$.  The couple $(y_t,v_t)$ is then distributed according to the  density law $\dens(t,u) dx du := \tfrac{dx}{|\Dd|}(G_{\sigma}(t)* \tfrac{\overline{P}_0(|\cdot|)}{\| \overline{P}_0\|_1}) du$. Indeed, one  can easily check that $\dens$ is a solution to the Fokker-Planck equation  \eqref{eq:ConditionalMcKeanVlasov-Pde} in $L^2((0,T)\times\Dd;H^1(\er^d))\cap V_1(\weight;Q_T)$ with $B=0$. 
The combination of Propositions 4.1 and 4.2 in \cite{BoJa-15} (in the case $B=0$)  gives  $\QQ\circ (y_t,v_t)^{-1}  = \dens(t,u) dx du$.  Applying the Cauchy-Schwarz inequality, we get
\begin{align*}
&\EE_{\QQ}\left[\sqrt{\weight(\UUU_t)}\psi(X_t,\UUU_t)\right] \leq |\Dd| \exp\{
\tfrac{\Vert V\Vert^2_{\infty}T}{4\sigma^2}\}
\sqrt{\EE_{\QQ}\left[\weight(v_t)\psi^2(y_t,v_t)\right] }
\leq C \sqrt{ \int_{\Dd\times\er^d}\psi^2(x,u)\weight(u) \dens(t,u) dx\,du}.
\end{align*}
Observing that $\weight(u)\leq 2^{\alpha-1}(\weight(u-v)+\weight(v))$, for all $u,v\in\er^d$, we get
\begin{align*}
\weight(u)\dens(t,u) &\leq 2^{\alpha-1}\left(\int_{\er^d}\weight(u-v)G_\sigma(t,u-v)\overline{P}_0(|v|)dv +\int_{\er^d}\weight(v)G_\sigma(t,u-v)\overline{P}_0(|v|)dv\right)
\leq  2^{\alpha-1} C(\alpha)
\end{align*}
with  $C(\alpha) = \|\weight G_\sigma(t,\cdot)\|_{\infty}\|\overline{P}_0\|_1 +\|G_\sigma(t,\cdot)\|_{\infty}\|\weight\overline{P}_0\|_1$, that allows to conclude on \eqref{eq:weightedL2bound}.

Now we prove the probabilistic interpretation of trace integrals in \eqref{eq:StochasticInterpretationTrace}. We consider the unique solution $\rho$ in $V_1(\weight;Q_T)$ and $\gamma^{\pm}(\rho)$ in
$L^{2}(\weight;\Sigma^{\pm}_T)$, of the following weak Fokker-Planck equation starting from $\rho_0$:  for all $t\in(0,T]$, $\phi\in\Cc^\infty_c([0,t]\times\overline{\Dd}\times\er^d)$, we have
\begin{equation}\label{eq:GreenFormula}
\begin{aligned}
& \int_{\Sigma^+_t}
\left(u\cdot n_{\Dd}(x)\right)\gamma^+(\rho)(s,x,u)\phi(s,x,u)\,d\lambda_{\Sigma_T}+ \int_{\Sigma^-_{t}}
\left(u\cdot n_{\Dd}(x)\right)\gamma^-(\rho)(s,x,u)\phi(s,x,u)\,d\lambda_{\Sigma_T}\\
&{\displaystyle = -\int_{\Dd\times\er^{d}}\phi(t,x,u)\rho(t,x,u)\,dx\,du
+\int_{\Dd\times\er^{d}}\phi(0,x,u)\rho_{0}(x,u)\,dx\,du }\\
&\quad{\displaystyle +\int_{Q_{t}} \left(\partial_{s}\phi + u\cdot\nabla_x \phi+V\cdot \nabla_{u}\phi
+\frac{\sigma^2}{2}\triangle_u \phi\right)(s,x,u)\rho(s,x,u)\,ds\,dx\,du}\\
\end{aligned}
\end{equation}
and such that $\gamma^{\pm}(\rho)$ satisfy the Maxwellian bounds \eqref{eq:TraceBounds} (see \cite[Proposition 3.14]{BoJa-15}). Then it is straightforward to check that $\rho(t)$ is also the density of $(X_t,U_t)$  using the identification by mild-equation used in  \cite{BoJa-15}-Proposition 4.2.
Now  applying It{\^o}'s formula to $\EE[\phi(T,X_T,U_T)]$, combining with \eqref{eq:GreenFormula},  one has
\begin{align*}
&\EE\left[\sum_{0<s \leq t}\left(\phi(s,X_s,U_{s^-}-2(U_{s^-}\cdot\nd(X_s))\nd(X_s))
-\phi(s,X_s,U_{s^-})\right)\1_{\{X_s\in\partial\Dd\}}\right]\\
&=\int_{\Sigma^+_T}\left(u\cdot \nd(x)\right)\gamma^+(\rho)(s,x,u)\phi(s,x,u)\,d\lambda_{\Sigma_T}+\int_{\Sigma^-_T}
\left(u\cdot n_{\Dd}(x)\right)\gamma^-(\rho)(s,x,u)\phi(s,x,u)\,d\lambda_{\Sigma_T}.
\end{align*}
Using the density between $\Cc^\infty_c(\overline{\Dd}\times\er^d)$ and $\Cc_c(\overline{\Dd}\times\er^d)$ and the surjectivity of the application $\phi\in\Cc_c(\overline{\Dd}\times\er^d)\rightarrow \phi\big{|}_{\Sigma_T}in\Cc_c(\Sigma_T)$, we conclude on \eqref{eq:StochasticInterpretationTrace}.

\subsection{Proof of the $\PP(\tau_n \leq T)$'s  upper bound \eqref{eq:Correction}}\label{sec:A2}

We consider the Langevin process  on  the probability measure $\PP_{y,v}$, endowed with a Brownian motion $B$,
\[
x_t = y + \int_0^t u_s ds, \quad u_t = v + B_t,
\]
with $y\neq 0$,  and the sequence of passage times
\begin{align*}
\displaystyle \tau_{n}=\inf\{\tau_{n-1}< t \leq T; x_{t}=0\},\textrm{ for }n \geq 1,\quad \tau_{0}=0.
\end{align*}

From the expression of $\PP_{y,v}(\tau_n \in dt,|u_{\tau_n}|\in dz)$ given in Theorem 3, in Lachal \cite{Lachal1997}, we obtain that, {for any $n\geq 2$},
\begin{equation}\label{formula1}
\begin{aligned}
\PP_{y,v}(\tau_n \leq T)=\int_0^T dt\int_0^t \frac{ds}{s} \int_0^{+\infty} \frac{dz}{\pi^2}\int_0^{+\infty}du\, g(t-s,y,v;0,u)\exp
\Big(-\frac{2(z^2+u^2)}{s}\Big)\\
\times \left[\int_0^{+\infty}\frac{\gamma \sinh(\pi\gamma)}{(2\cosh(\frac{\pi\gamma}{3}))^{n-1}}K_{i\gamma}\Big(\frac{4 uz}{s}\Big{)}\,d\gamma-\int_0^{+\infty}\frac{\gamma\sinh(\pi\gamma)}{(2\cosh(\frac{\pi\gamma}{3}))^{n}}K_{i\gamma}\Big(\frac{4 uz}{s}\Big{)}d\gamma\right]
\end{aligned}
\end{equation}
where
\[
g(t-s,y,v;0,u)=\frac{2\sqrt{3}}{\pi (t-s)^2}\exp\Big(-\frac{6y^2}{(t-s)^3}-\frac{6yv}{(t-s)^2}-\frac{2(u^2+v^2)}{(t-s)}\Big)\cosh\Big(
\frac{2u}{(t-s)^2}(3y+(t-s)v)\Big)
\]
and
$K_{i\gamma}(a)=\int_0^{+\infty}\exp\left\{-a\cosh(t)\right\}\cos(\gamma t)\,dt$ is the modified Bessel function.
We then work with the expression in \eqref{formula1}, using the following tricky identity, successfully used  for similar computation in Profeta~\cite{Profeta-13}:
\begin{equation}
\label{formula2}
\sinh(\pi\gamma)=\sinh(\frac{\pi\gamma}{3})\left(4\cosh^2(\frac{\pi\gamma}{3})-1\right).
\end{equation}
Assuming now that $n\geq 5$, for the first integral in \eqref{formula1}, from \eqref{formula2}, we have the decomposition
\begin{equation}\label{formula3}
\begin{aligned}
\int_0^{+\infty}\frac{\gamma \sinh(\pi\gamma)}{(2\cosh(\frac{\pi\gamma}{3}))^{n-1}}K_{i\gamma}(\frac{4 uz}{s})\,d\gamma
&=\frac{1}{2^{n-3}}\int_0^{+\infty}\gamma K_{i\gamma}(\frac{4 uz}{s})\frac{\sinh(\frac{\pi\gamma}{3})}{\cosh(\frac{\pi\gamma}{3})^{n-3}}\,d\gamma\\
&\quad -\frac{1}{2^{n-1}}\int_0^{+\infty}\gamma K_{i\gamma}(\frac{4 uz}{s})\frac{\sinh(\frac{\pi\gamma}{3})}{\cosh(\frac{\pi\gamma}{3})^{n-1}}\,d\gamma.
\end{aligned}
\end{equation}
Furthermore, for all $k\geq 2$, $a\geq 0$, we have (see \cite{Profeta-13}, page $168$)
\begin{align*}
\int_0^{+\infty}\gamma K_{i\gamma}(a)\frac{\sinh(\frac{\pi\gamma}{3})}{\cosh(\frac{\pi\gamma}{3})^k}\,d\gamma
=\int_0^{+\infty}a \sinh(\theta)\exp\Big(-a\cosh(\theta)\Big)
\left(\int_0^{+\infty}\sin(\gamma\theta)\frac{\sinh(\frac{\pi\gamma}{3})}{\cosh(\frac{\pi\gamma}{3})^k}\,d\gamma\right)\,d\theta.
\end{align*}
Since $|\sin(\gamma\theta)\frac{\sinh(z)}{\cosh(z)}|\leq 1$,  and $\cosh \geq 1$, assuming $k\geq 2$, we have
\[
\left|\int_0^{+\infty}\sin(\gamma\theta)\frac{\sinh(\frac{\pi\gamma}{3})}{\cosh(\frac{\pi\gamma}{3})^k}\,d\gamma\right|
\leq \int_0^{+\infty}\frac{d\gamma}{\cosh(\frac{\pi\gamma}{3})} = \frac{3}{2},
\]
and since $\sinh\geq 0$ on $\er^+$,
\begin{align*}
\left|\int_0^{+\infty}\gamma K_{i\gamma}(a)\frac{\sinh(\frac{\pi\gamma}{3})}{\cosh(\frac{\pi\gamma}{3})^{k}}\,d\gamma\right|
&\leq \frac{3}{2}\int_0^{+\infty}a\sinh(\theta)\exp\big(-a\cosh(\theta)\big) d\theta
=\frac{3}{2}\exp(-a).
\end{align*}
By taking $a=\frac{4 uz}{s}\geq 0$ in the preceding expression and coming back to \eqref{formula3},  we deduce that
\[
\left| \int_0^{+\infty}\frac{\gamma\sinh(\pi\gamma)}{(2\cosh(\frac{\pi\gamma}{3}))^{n-1}}K_{i\gamma}\big(\frac{4uz}{s}\big)\,d\gamma\right|
\leq \frac{8}{2^{n-1}}\exp\big(-\frac{4uz}{s}\big).
\]
In the same way, we get
\[
\left| \int_0^{+\infty}\frac{\gamma\sinh(\pi\gamma)}{(2\cosh(\frac{\pi\gamma}{3}))^{n}}K_{i\gamma}\Big(\frac{4uz}{s}\big)\,d\gamma\right|
\leq \frac{8}{2^n}\exp\big(-\frac{4uz}{s}\big).
\]
Therefore, coming back to \eqref{formula1}, as $\int_0^{+\infty}\exp(-\tfrac{2(z^2+u^2)}{s})\exp(-\tfrac{4uz}{s}) \,dz
\leq \exp(-\tfrac{2 u^2}{s})\int_{u}^{+\infty}\exp(-\tfrac{2 z^2}{s})\,dz \leq \sqrt{s}$, we have
\[
\PP_{y,v}(\tau_n\leq T)\leq \frac{3}{2^{n-3}\pi^2}\int_0^T dt \int_0^t \frac{ds}{\sqrt{s}} \int_0^{+\infty}g(t-s,y,v;0,u)\,du.
\]
Let us now bound the integral
\begin{align*}
\int_0^{+\infty}g(t-s,y,v;0,u)\,du &=\frac{\sqrt{3}}{\pi (t-s)^2}\exp\Big(
-\frac{6y^2}{(t-s)^3}-\frac{6yv}{(t-s)^2}-\frac{2v^2}{(t-s)}\Big)\\
&\quad \times\int_0^{+\infty}\exp\Big(
-\frac{2u^2}{(t-s)}\Big)\left(\exp\Big(\frac{2u(3y+(t-s)v)}{(t-s)^2}\Big)+\exp\Big(\frac{-2u(3y+(t-s)v)}{(t-s)^2}\Big)\right)\,du\\
&\quad \leq \frac{\sqrt{3}}{\sqrt{2\pi(t-s)^3}}\exp\Big(
-\frac{6y^2}{(t-s)^3}-\frac{6yv}{(t-s)^2}-\frac{2v^2}{(t-s)}\Big)\exp\Big(-\frac{(3y+(t-s)v)^2}{2(t-s)^3}\Big)\\
&\quad\leq\frac{\sqrt{3}}{\sqrt{2\pi(t-s)^3}}\exp\Big(-\frac{3(y+(t-s)v)^2}{2(t-s)^3}\Big),
\end{align*}
so that,
\begin{align*}
\PP_{y,v}(\tau_n \leq T)
&\leq \frac{3}{2^{n-3}\pi^2}\int_0^Tdt \int_0^t\frac{ds}{\sqrt{s}} \frac{\sqrt{3}}{ \sqrt{2{\pi}(t-s)^3}}\exp\Big(-\frac{3(y+(t-s)v)^2}{2(t-s)^3}\Big)\\
&\quad = \frac{3}{2^{n-3}\pi^2}
\int_0^T\frac{ds}{\sqrt{s}}\int_0^{T-s}\frac{\sqrt{3}}{\sqrt{2\pi} \sqrt{t^3}}\exp\Big(-\frac{3(y+tv)^2}{2t^3}\Big)dt\\
&\leq \frac{3\sqrt{3}\sqrt{T}}{2^{n-4}\pi^2\sqrt{2\pi}}\int_0^{T}\frac{1}{\sqrt{t^3}}\exp\Big(-\frac{3(y+tv)^2}{2t^3}\Big)dt.
\end{align*}
Since $\beta^* = \sup \{\beta >0;  \text{supp}(\mu_0)  \subset[\beta,+\infty)\times\er \}>0$, we observe that for $t\leq\beta^*/2|v|$, $y+tv\geq \beta^*-t|v|\geq \tfrac{\beta^*}{2} >0$  and
\begin{align*}
\int_0^{T}\frac{1}{\sqrt{t^3}}\exp\Big(-\frac{3(y+tv)^2}{2t^3}\Big)dt
& \leq \int_0^{T\wedge \tfrac{\beta^*}{2|v|}}\frac{1}{\sqrt{t^3}}
\exp\Big(-\frac{3(\beta^*)^2}{8t^3}\Big)dt
+\int_{T\wedge\tfrac{\beta^*}{2|v|}}^T\frac{1}{\sqrt{t^3}}dt \\
&\leq\frac{\sqrt{2}T}{\sqrt{3}\beta^*}  +\frac{1}{2}\left(\frac{1}{\sqrt{T\wedge\tfrac{\beta^*}{2|v|}}}-\frac{1}{\sqrt{T}}\right) \leq \frac{\sqrt{2}T}{\sqrt{3}\beta^*}   +\sqrt{\frac{|v|}{2\beta^*}}.
\end{align*}
This estimate implies that, for some constant $C>0$,
\[
\PP_{y,v}(\tau_n\leq T)\leq \frac{C}{2^{n}}
\left(\frac{\sqrt{2} T}{\sqrt{3}\beta^*}   +\sqrt{\frac{|v|}{2\beta^*}}\right),
\]
and further we obtain the desired upper bound \eqref{eq:Correction}, 
\[
\PP(\tau_n\leq T)=\int \PP_{y,v}(\tau_n\leq T)\mu_0(dy,dv)\leq \frac{C(T,m_0,\beta^*)}{2^n}.
\]


\begin{thebibliography}{1}

\bibitem{jabir-10b}
F.~Bernardin, M.~Bossy, C.~Chauvin, J.-F. Jabir, and A.~Rousseau.
\newblock {Stochastic Lagragian method for downscaling problems in
  computational fluid dynamics}.
\newblock {\em M2AN Math. Model. Numer. Anal.}, 44(5):885--920, 2010.

\bibitem{BoJa-11}
M.~Bossy and J.-F. Jabir.
\newblock {On confined McKean Langevin processes satisfying the mean
  no-permeability condition}.
\newblock {\em Stochastic Process. Appl.}, 121:2751--2775, 2011.

\bibitem{BoJa-15}
M.~Bossy and J.-F. Jabir.
\newblock {Lagrangian stochastic models with specular boundary condition}.
\newblock {\em Journal of Functional Analysis}, 268(6):1309 -- 1381, 2015.

\bibitem{BoJaTa-10}
M.~Bossy, J.-F. Jabir, and D.~Talay.
\newblock {On conditional McKean Lagrangian stochastic models}.
\newblock {\em Probab. Theory Relat. Fields}, 151(1--2):319--351, 2011.

\bibitem{JaSh-02}
J.~Jacod and A.~N.~Shiryaev.
\newblock {\em Limit Theorems for Stochastic Processes}.
\newblock Second Edition. Springer, New York, 2002.

\bibitem{Lachal1997}
A.~Lachal.
\newblock {Les temps de passage successifs de l'int\'egrale du mouvement  brownien}.
\newblock {\em Ann. Inst. H. Poincar\'e Probab. Statist.}, 33(1):1--36, 1997.

\bibitem{Profeta-13}
C.~Profeta.
\newblock Some limiting laws associated with the integrated Brownian motion,  2013.
\newblock {\em ESAIM: Probability  Statistics}, 19: 148-171, 2015.

\bibitem{ASznitman1984}
A.-S. Sznitman.
\newblock {Nonlinear reflecting diffusion process, and the propagation of chaos  and fluctuations associated}.
\newblock {\em J. Funct. Anal.}, 56(3):311--336, 1984.
\end{thebibliography}

\end{document}